\documentclass{article}
\usepackage{mathrsfs}
\usepackage{amsfonts}
\usepackage{amssymb,amsmath}
\usepackage{amsthm}
\usepackage[dvips]{graphics}
\usepackage{ae}
\usepackage{bm}
\usepackage{graphicx}
\usepackage{cite}
\usepackage{lineno}

\def\rank{\mbox{\rm rank}}
\def\corank{\mbox{\rm corank}}

\def\Int{\mbox{\rm Int}}
\def\dim{\mbox{\rm dim}}

\def\max{\mbox{\rm max}}

\def\And{\mbox{\rm ~and~}}

\def\Int{\mbox{\rm Int}}

\def\L{\mbox{\rm L}}

\def\({\mbox{\rm (}}\def\){\mbox{\rm )}}

\def\sp{\hspace{1ex}}
\def\sp{\hspace{0.3cm}}

\newtheorem{theorem}{Theorem}[section]
\newtheorem{proposition}[theorem]{Proposition}

\newtheorem{lemma}[theorem]{Lemma}

\begin{document}
\normalsize
\title{\huge \textbf{M\"{o}bius Conjugation and Convolution Formulae}}

\author{
    Suijie Wang\\
    Institute of Mathematics\\
    Academia Sinica, Taiwan\\
    wangsuijie@math.sinica.edu.tw
}
\date{\today}
\maketitle

\begin{abstract}

Let $P$ be a locally finite poset with the interval space $\Int(P)$, and $R$ a ring with identity. We shall introduce the M\"{o}bius conjugation $\mu^\ast$ sending each function $f:P\to R$ to an incidence function $\mu^\ast(f):\Int(P)\to R$ such that $\mu^\ast(fg)=\mu^\ast(f)\ast\mu^\ast(g)$. Taking $P$ to be the intersection poset of a hyperplane arrangement $\mathcal{A}$, we shall obtain a convolution identity for the number $r(\mathcal{A})$ of regions and the number $b(\mathcal{A})$ of relatively bounded regions, and a reciprocity theorem of the characteristic polynomial $\chi(\mathcal{A},t)$, which also leads to a combinatorial interpretation to the values $|\chi(\mathcal{A},-q)|$ for large primes $q$. Moreover, all known convolution identities on Tutte polynomials of matroids will be direct consequences after specializing the poset $P$ and functions $f,g$. 
\end{abstract}

\section{M\"{o}bius Conjugation}
We use the definitions and notations of posets from \cite{Stanley1}, and all posets in this paper are assumed to be locally finite. Let $P$ be a poset and $R$ a ring with identity. Denote by $\Int(P)$ the interval space of $P$ and $\mathcal{I}(P,R)=\{\alpha:\Int(P)\to R\}$ the incidence algebra of $P$ whose multiplication structure is given by the convolution product, i.e., for any $\alpha, \beta \in \mathcal{I}(P)$ and $x\le y$ in $P$,
\[
[\alpha\ast \beta](x,y)=\sum_{x\le z\le y}\alpha(x,z)\,\beta(z,y),\sp \forall\, \alpha,\beta\in \mathcal{I}( P).
\]
Let $R^P$ be the ring of all functions $f:P\to R$ whose ring structure is given by point-wise multiplication and addition. Define the \emph{M\"{o}bius conjugation} $\mu^\ast: R^P \to \mathcal{I}(P, R)$ to be
\[\mu^\ast(f)=\mu\ast\delta(f)\ast\zeta,\sp \forall\; f\in R^P.\]
where $\mu$ is the M\"{o}bius function of $P$, and the map $\delta: R^P \to \mathcal{I}(P, R)$ is defined by $\delta(f)(x,y)=f(x)$ if $x=y$ and $0$ otherwise, for all $f\in R^P$ and $x\le y$ in $P$.

\begin{theorem}\label{Mobius-conj}
With above settings, the map $\mu^\ast$ is a ring monomorphism, i.e.,
\[
\mu^\ast(fg)=\mu^\ast(f)\ast\mu^\ast(g),\sp\forall\; f,g\in R^P.
\]
\end{theorem}
\begin{proof}Given $f,g\in R^P$, it is obvious that $\delta(fg)=\delta(f)\ast\delta(g)$. Then
\begin{eqnarray*}
\mu^\ast(fg)=\mu\ast\delta(f)\ast\zeta\ast\mu\ast\delta(g)\ast\zeta=\mu^\ast(f)\ast\mu^\ast(g).
\end{eqnarray*}
So $\mu^\ast$ is a homomorphism as rings. To prove the injectivity, suppose $\mu^\ast(f)=\mu^\ast(g)$ for some $f,g\in R^P$. Multiplying $\zeta$ on the left hand side and $\mu$ on the right hand side respectively, we obtain that $\delta(f)=\delta(g)$. Thus $f=g$.
\end{proof}
Multiplicative identities for chromatic polynomials first appeared in \cite{Tutte1967} by W. Tutte in 1967. In 1999, W. Kook, V. Riener and D. Stanton \cite{Staton1999} found a convolution formula for Tutte polynomial of matroids. After that, Joseph P. S. Kung \cite{Joseph2004} gave a multiplicative identities for characteristic polynomials of matroids in 2004. And he also formulated many generalizations of all previous identities in 2010 \cite{Joseph2010}. We shall see that Theorem \ref{Mobius-conj} gives the algebraic essence to all these known identities. In fact, by specializing the poset $P$ and the functions $f,g$ of Theorem \ref{Mobius-conj}, we will obtain all those identities.

\section{Convolution Formula on Characteristic Polynomials}
A hyperplane arrangement $\mathcal{A}$ in a vector space $V$ is a collection of finite hyperplanes of $V$. The intersection semi-lattice $L(\mathcal{A})$ of $\mathcal{A}$ is defined to be the collection of all nonempty intersections of hyperplanes in $\mathcal{A}$, whose partial order is given by the inverse of set inclusion. Namely,
\[
L(\mathcal{A})=\{\cap_{H\in\mathcal{B}}H\mid \mathcal{B}\subseteq \mathcal{A}\},
\]
whose minimal element is $\hat{0}=\cap_{H\in \emptyset} H:=V\in L(\mathcal{A})$. Artificially adding a maximal element $\hat{1}=\emptyset$ to $L(\mathcal{A})$, $L(\mathcal{A})$ then becomes a geometric lattice, denoted $L^{\ast}(\mathcal{A})=L(\mathcal{A})\cup \{\hat{1}\}$ and called the \emph{reduced intersection lattice} of $\mathcal{A}$. With the assumptions $\dim(\hat{1})=\infty$ and $t^\infty=0$, the \emph{characteristic polynomial} $\chi(\mathcal{A},t)\in \Bbb{C}[t]$ of $\mathcal{A}$ can be written as
\[
\chi(\mathcal{A},t):=\sum_{X\in L(\mathcal{A})}\mu(\hat{0},X)\,t^{{\small\dim(X)}}=\sum_{X\in L^{\ast}(\mathcal{A})}\mu(\hat{0},X)\,t^{{\small\dim(X)}}.
\]
Given $X,Y\in L^{\ast}(\mathcal{A})$ and $X\leq Y$, let $\mathcal{A}_{X,Y}$ be a hyperplane arrangement in the vector space $X$ defined by
\[
\mathcal{A}_{X,Y}=\{H \cap X \mid H\in \mathcal{A}\,\,\text{with}\,\, Y \subseteq H \And\,X\nsubseteq H\}.
\]
In particular, $\mathcal{A}_{X,X}=\emptyset$. If $X,Y\in L(\mathcal{A})$ and $X\le Y$, it is easily seen that $L(\mathcal{A}_{X,Y})\cong [X,Y]$ as lattices.
It follows that the M\"{o}bius function of $L^\ast(\mathcal{A}_{X,Y})$ is the same as the restriction of the M\"{o}bius function of $L^\ast(\mathcal{A})$ onto the interval $[X,Y]$. Then the characteristic polynomial of $\mathcal{A}_{X,Y}$ is
\[
\chi(\mathcal{A}_{X,Y},t)=\sum_{X\le Z\le Y}\mu(X,Z)\,t^{{\small \dim(Z)}},
\]
where $\mu$ is the M\"obius function of $L^\ast(\mathcal{A})$. In particular, $\chi(\mathcal{A}_{\hat{0},\hat{1}},t)=\chi(\mathcal{A},t)$ and $\chi(\mathcal{A}_{\hat{1},\hat{1}},t)=0.$
For convenience, we denote, for any $X\in L(\mathcal{A})$,
\begin{eqnarray*}
\mathcal{A}|X&=&\mathcal{A}_{\hat{0},X}=\{H\in \mathcal{A}\mid X\subseteq H\},\\
\mathcal{A}/X&=&\mathcal{A}_{X,\hat{1}}=\{H\cap X\mid H\in \mathcal{A}-\mathcal{A}|X\}.
\end{eqnarray*}
\begin{theorem}\label{arragement}
Let $\mathcal{A}$ be a hyperplane arrangement with the reduced intersection lattice $L^\ast(\mathcal{A})$. If $X\le Y$ in $L^\ast(\mathcal{A})$, then we have
\begin{equation*}
\chi(\mathcal{A}_{\small X,Y},st)=\sum_{Z\in L^\ast(\mathcal{A});X\le Z\le Y}\chi(\mathcal{A}_{X,Z},s)\,\chi(\mathcal{A}_{Z,Y},t).
\end{equation*}
Taking $X=\hat{0}$ and $Y=\hat{1}$ in particular, we have
\begin{equation}\label{arr-conv}
\chi(\mathcal{A},st)=\sum_{X\in L(\mathcal{A})}\chi(\mathcal{A}|X,s)\,\chi(\mathcal{A}/X,t).
\end{equation}
\begin{proof}Define $f,g:L^\ast(\mathcal{A})\to \Bbb{C}[s,t]$ to be $f(X)=t^{\dim(X)}$ and $g(X)=s^{\dim(X)}$. Then for any $X\le Y$ in $\L^\ast(\mathcal{A})$, we have
\[\mu^\ast(f)(X,Y)=\left[\mu\ast \delta(f)\ast \zeta\right](X,Y)=\sum_{X\le Z\le Y}\mu(X,Z)t^{\dim(Z)}=\chi(\mathcal{A}_{X,Y},t).\]
Similarly, $\chi(\mathcal{A}_{X,Y},s)=\mu^\ast(g)(X,Y)$ and $\chi(\mathcal{A}_{X,Y},st)=\mu^\ast(gf)(X,Y)$. Applying Theorem \ref{Mobius-conj}, then
\[\chi(\mathcal{A}_{X,Y},st)=\left[\mu^\ast(g)\ast\mu^\ast(f)\right](X,Y)=\sum_{Z\in L^\ast(\mathcal{A});X\le Z\le Y}\mu^\ast(g)(X,Z)\mu^\ast(f)(Z,Y).\]
This completes the proof.
\end{proof}
\end{theorem}
Joseph P. S. Kung \cite{Joseph2004} found the convolution formula (\ref{arr-conv}) for characteristic polynomials of matroids. Here the formula (\ref{arr-conv}) of Theorem \ref{arragement} extends it to affine hyperplane arrangements, which are not necessarily a matroid. In the next two subsections, we shall give two combinatorial identities by applying the formula (\ref{arr-conv}).

\subsection{Convolution Formula on $r(\mathcal{A})$ and $b(\mathcal{A})$}
If $\mathcal{A}$ is a hyperplane arrangement in the real vector space $V=\Bbb{R}^n$, its complement $M(\mathcal{A})=V-\cup_{H\in\mathcal{A}}H$ consists of finite many disjoint connected components, called \emph{regions} of $\mathcal{A}$. If $\dim(V)=n$, denote by $\mathscr{R}(\mathcal{A})$ the collection of all regions of $M(\mathcal{A})$ and $r(\mathcal{A})=\#\mathscr{R}(\mathcal{A})$.
Let $W$ be the subspace spanned by the normal vectors of $H$ for all $H\in \mathcal{A}$. A region $\Delta\in \mathscr{R}(\mathcal{A})$ is called \emph{relatively bounded} if $\Delta\cap W$ is bounded in $W$. Denote by $\mathscr{B}(\mathcal{A})$ the collection of all relatively bounded regions of $M(\mathcal{A})$ and $b(\mathcal{A})=\#\mathscr{B}(\mathcal{A})$. Zaslavski formula \cite{Zaslavsky} states that
\begin{eqnarray*}
r(\mathcal{A})=(-1)^{\small \dim(V)}\chi(\mathcal{A},-1),\sp b(\mathcal{A})=(-1)^{{\small \rank(\mathcal{A})}}\chi(\mathcal{A},1),
\end{eqnarray*}
where $\rank(\mathcal{A})=\max\{\rank(X)\mid X\in L(\mathcal{A})\}$ and $\rank(X)=\dim(V)-\dim(X)$. Then the following convolution formula of $r(\mathcal{A})$ and $b(\mathcal{A})$ can be easily obtained from (\ref{arr-conv}).
\begin{theorem}\label{region-conv}Denote by $\corank(X)=\rank(\mathcal{A})-\rank(X)$. Then
\begin{eqnarray*}
b(\mathcal{A})&=&\sum_{X\in L(\mathcal{A})}(-1)^{{\small \corank(X)}}\,r(\mathcal{A}|X)\,r(\mathcal{A}/X),\\
r(\mathcal{A})&=&\sum_{X\in L(\mathcal{A})}(-1)^{{\small \corank(X)}}\,r(\mathcal{A}|X)\,b(\mathcal{A}/X).
\end{eqnarray*}
\end{theorem}
We next give a combinatorial proof to above identities. The idea is to consider the total signs on the right side contributed to each region on the left side of the identities. If $\mathcal{A}$ is a hyperplane arrangement in $\Bbb{R}^n$, recall that $W$ is the subspace spanned by the normal vectors of all hyperplanes in $\mathcal{A}$. Consider the arrangement $\mathcal{A}_W=\{H\cap W\mid H\in \mathcal{A}\}$ in $W$ whose rank $r(\mathcal{A}_W)=\dim(W)$. It is easily seen that $r(\mathcal{A})=r(\mathcal{A}_W)$, $b(\mathcal{A})=b(\mathcal{A}_W)$, and $L(\mathcal{A})\cong L(\mathcal{A}_W)$. Assume that $\mathcal{A}$ is a hyperplane arrangement in $\Bbb{R}^n$ with rank $r(\mathcal{A})=n$. Then all relatively bounded regions in $M(\mathcal{A})$ are actually bounded and $\corank(X)=\dim(X)$ for all $X\in L(\mathcal{A})$. With the induced topology of the standard topology of $\Bbb{R}^n$, each region $\Delta\in \mathscr{R}(\mathcal{A})$ is homeomorphic to an open ball of dimension $n$. Given any $\Delta\in \mathscr{R}(\mathcal{A})$, let $\bar{\Delta}$ be its topological closure and denote by $F(\Delta)$ the collection of all \emph{faces} of $\bar{\Delta}$, which is defined to be
\[
F(\Delta)=\{\bar{\Delta}\cap M(\mathcal{A}/X)\mid X\in L(\mathcal{A}),\bar{\Delta}\cap M(\mathcal{A}/X)\neq \emptyset\}.
\]
If $\Delta\in \mathscr{B}(\mathcal{A})$ is a bounded region in $M(\mathcal{A})$, $\bar{\Delta}$ becomes a closed polytope and is homeomorphic to a closed ball. Then the Euler characteristic of $\bar{\Delta}$ is $1$, i.e.,
\begin{eqnarray}\label{F}
\sum_{f\in F(\Delta)}(-1)^{\dim(f)}=1,\sp \forall\, \Delta\in \mathscr{B}(\mathcal{A}).
\end{eqnarray}
If $\Delta\in \mathscr{R}(\mathcal{A})-\mathscr{B}(\mathcal{A})$ is an unbounded region in $M(\mathcal{A})$, denote by $F_b(\Delta)$ the collection of bounded faces of $\bar{\Delta}$. Then the space $\sqcup_{f\in F_b(\Delta)}f$ is homeomorphic to a closed ball, i.e.,
\begin{equation}\label{F_b}
\sum_{f\in F_b(\Delta)}(-1)^{\dim(f)}=1,\sp \forall\;\Delta\in \mathscr{R}(\mathcal{A})-\mathscr{B}(\mathcal{A}).
\end{equation}
On the other hand, when $\Delta\in \mathscr{R}(\mathcal{A})-\mathscr{B}(\mathcal{A})$ is an unbounded region in $M(\mathcal{A})$, $\bar{\Delta}$ is homeomorphic to a closed half space whose Euler characteristic is 0.
Then we have
\begin{eqnarray}\label{F_u}
\sum_{f\in F(\Delta)}(-1)^{\dim(f)}=0.
\end{eqnarray}
Applying (\ref{F}), (\ref{F_b}), and (\ref{F_u}), we have
\begin{eqnarray*}
r(\mathcal{A})=\sum_{\Delta\in \mathscr{B}(\mathcal{A})}\sum_{f\in F(\Delta)}(-1)^{\dim(f)}+\sum_{\Delta\in \mathscr{R}(\mathcal{A})-\mathscr{B}(\mathcal{A})}\;\sum_{f\in F_b(\Delta)}(-1)^{\dim(f)}=\sum_{\Delta\in \mathscr{R}(\mathcal{A})}\sum_{f\in F_b(\Delta)}(-1)^{\dim(f)},
\end{eqnarray*}
\begin{eqnarray*}b(\mathcal{A})=\sum_{\Delta\in \mathscr{B}(\mathcal{A})}\sum_{f\in F(\Delta)}(-1)^{\dim(f)}+\sum_{\Delta\in \mathscr{R}(\mathcal{A})-\mathscr{B}(\mathcal{A})}\;\sum_{f\in F(\Delta)}(-1)^{\dim(f)}=\sum_{\Delta\in \mathscr{R}(\mathcal{A})}\sum_{f\in F(\Delta)}(-1)^{\dim(f)}.
\end{eqnarray*}
From the definition of $F(\Delta)$, we can see that for each face $f\in F(\Delta)$, there exists a unique $X\in L(\mathcal{A})$ such that $f\subseteq X$ and $\dim(f)=\dim(X)$, i.e., $f\in \mathscr{R}(\mathcal{A}/X)$. However, for each $f\in\mathscr{R}(\mathcal{A}/X)$, there are $r(\mathcal{A}|X)$ regions $\Delta\in \mathscr{R}(\mathcal{A})$ such that $f\in F(\Delta)$. In addition, $f$ is bounded if and only if $f\in \mathscr{B}(\mathcal{A}/X)$. Then we have
\begin{eqnarray*}
r(\mathcal{A})&=&\sum_{X\in L(\mathcal{A})}(-1)^{{\small \dim(X)}}\,r(\mathcal{A}|X)\,b(\mathcal{A}/X),\\
b(\mathcal{A})&=&\sum_{X\in L(\mathcal{A})}(-1)^{{\small \dim(X)}}\,r(\mathcal{A}|X)\,r(\mathcal{A}/X).
\end{eqnarray*}
According to the assumption that $\mathcal{A}$ is essential in $\Bbb{R}^n$, we have $\corank(X)=\dim(X)$ for all $X\in L(\mathcal{A})$ which completes the proof.

\subsection{Reciprocity Theorem of Characteristic polynomials}
In this subsection, the hyperplane arrangement $\mathcal{A}$ is assumed to be \emph{integral}, i.e., each hyperplane $H\in\mathcal{A}$ is defined by an integral linear equation as follows
\begin{equation}\label{eqn-H}
H:\, a_1x_1+\cdots+a_nx_n=b,\sp b,a_i\in \Bbb{Z},\, 1\le i\le n.
\end{equation}
For any prime number $q$, the above hyperplane $H$ is automatically reduced to a hyperplane $H_q$ in $\Bbb{F}_q^n$, which is defined by the equation
\[H_q:\, a_1x_1+\cdots+a_nx_n=b\mod q,\]
called the $q$-reduction of $H$. Then $\mathcal{A}_q=\{H_q\mid H\in \mathcal{A}\}$ defines a hyperplane arrangement in the space $V=\Bbb{F}_q^n$. In this subsection, we always assume $q$ is large enough such that that the intersection semilattices $L(\mathcal{A})$ and $L(\mathcal{A}_q)$ are isomorphic, where the isomorphism is given by
\[X\mapsto X_q:=\cap_{H\in \mathcal{A}|X}H_q.\]
Similar as before, we have the following notations,
\[\mathcal{A}_q|X_q=\{H_q\mid H\in \mathcal{A}|X\},\sp \mathcal{A}_q/X_q=\{H_q\cap X_q\mid H\in \mathcal{A}/X\},\sp M(\mathcal{A}_q)=\Bbb{F}_q^n-\cup_{H\in \mathcal{A}}H_q.\]
C. A. Athanasiadis gave the following combinatorial interpretation of $\chi(\mathcal{A},q)$.
\begin{theorem}{\rm\cite{Athanasiadis}}\label{F_q}
Let $\mathcal{A}$ be an integral arrangement in $\Bbb{R}^n$ and $q$ be a large prime number. Then
\[
\chi(\mathcal{A},q)=\big|M(\mathcal{A}_q)\big|.
\]
\end{theorem}
It follows by $(\ref{arr-conv})$ that $\chi(\mathcal{A},-q)=\sum_{X\in L(\mathcal{A})}\chi(\mathcal{A}|X,-1)\,\chi(\mathcal{A}/X,q).$ Applying Zaslavski formula $r(\mathcal{A}|X)=(-1)^n\chi(\mathcal{A}|X,-1)$, we then obtain the following result which will leads to a combinatorial interpretation to the number $\chi(\mathcal{A},-q)$ for any large prime number $q$, known as the reciprocity theorem of the characteristic polynomial.
\begin{proposition}\label{prop}
With the same assumptions as {\rm Theorem \ref{F_q}}, we have
\[
\chi(\mathcal{A},-q)=(-1)^n\sum_{X\in L(\mathcal{A})}r(\mathcal{A}|X)\,\big|M(\mathcal{A}_q/X_q)\big|.
\]
\end{proposition}
To state the combinatorial aspect of above formula, we introduce some notations first. Fixing a large prime number $q$, denote by
\[
C(n,q)=\left\{(x_1,\ldots, x_n)\in \Bbb{Z}^n\mid -\frac{q-1}{2}\le x_i\le \frac{q-1}{2},\; \forall\; 1\le i\le n\right\}
\]
the central symmetric lattice cube of size $q$ and dimension $n$. Suppose $\mathcal{A}$ is an integral hyperplane arrangement and $H\in \mathcal{A}$ is defined by the equation (\ref{eqn-H}). For any $t\in \Bbb{R}$, let $H(t)$ be a hyperplane translated from $H$, whose defining equation is
\[
H(t):\sp \, a_1x_1+\cdots+a_nx_n=b+t.
\]
For any large prime number $q$, denote by $M(\mathcal{A},q)$ the complement of the union of $H(kq)$ for all $H\in\mathcal{A}$ and $k\in \Bbb{Z}$, and $\mathscr{R}(\mathcal{A},q)$ the collection of all connected components (regions) of $M(\mathcal{A},q)$, i.e.,
\[
M(\mathcal{A},q)=\Bbb{R}^n-\cup_{H\in \mathcal{A},k\in \Bbb{Z}}H(kq),\sp \mathscr{R}(\mathcal{A},q)=\left\{\Delta\mid \Delta \text{~ is a region of ~}M(\mathcal{A},q)\right\}.
\]
Now we are ready to state the reciprocity theorem for the characteristic polynomial of hyperplane arrangements, where the proof will be given later.
\begin{theorem}\label{comb-interpre}{\rm [Reciprocity Theorem]} Let $\mathcal{A}$ be an integral arrangement with the characteristic polynomial $\chi(\mathcal{A},t)$. For any large prime $q$ and $\Delta\in\mathscr{R}(\mathcal{A},q)$, denote by $\bar{\Delta}$ the topological closure of $\Delta$ and define
\[\bar{\chi}(\mathcal{A}, q):=\sum_{\Delta\in\mathscr{R}(\mathcal{A},q)}\big|\bar{\Delta}\cap C(n,q)\big|.\]
Then  $\bar{\chi}(\mathcal{A},q)=(-1)^n\chi(\mathcal{A},-q)$. Namely, $|\chi(\mathcal{A},-q)|$ counts the total number of lattice points in $\bar{\Delta}\cap C(n,q)$ for all $\Delta\in\mathscr{R}(\mathcal{A},q)$.
\end{theorem}
It should be noted that some other reciprocity theorems concerning characteristic polynomials of hyperplane arrangements have been studied. Stanley \cite{Stanley3} introduced a reciprocity law for the chromatic polynomials of graphs. Athanasiadis \cite{Athanasiadis2010} found a reciprocity law for the characteristic polynomial of a deformed linear arrangement, whose specialization on $m=1$ provides a different interpretation to $\chi(\mathcal{A},-q)$. As a direct consequence of standard Ehrhart theory, Beck and Zaslavski \cite{Beck} generalized the reciprocity law of Ehrhart quasi-polynomials to a convex polytope dissected by a hyperplane arrangement. We shall see how Theorem \ref{comb-interpre} is related to Beck and Zaslavsky's reciprocity theorem, and then connected to the Ehrhart theory in the next paragraph.

Assume that $\mathcal{A}$ is a linear arrangement, i.e., the defining equations of all hyperplanes in $\mathcal{A}$ are homogeneous. With this assumption, denote by $\mathcal{A}^\ast=\{H(k)\mid H\in \mathcal{A}, k\in \Bbb{Z}\}$ the deformed arrangement of $\mathcal{A}$. Let $P=(-\frac{1}{2}, \frac{1}{2})^n\subseteq \Bbb{R}^n$. Then $P\setminus \bigcup_{H\in \mathcal{A}, k\in \Bbb{Z}}H(k)$ consists of finite many open rational polytopes, denoted by $R_1,\ldots, R_l$. Let
\[
E_{P,\mathcal{A}}(q)=\sum_{i=1}^{l}|qR_i\cap \Bbb{Z}^n|, \sp  \bar{E}_{P,\mathcal{A}}(q)=\sum_{i=1}^{l}|q\bar{R_i}\cap \Bbb{Z}^n|.
\]
Note that for all $x\in \Bbb{Z}^n$, $q^{-1}x\in H(k)\Leftrightarrow x\in H(kq)$, and $q^{-1}x\in P \Leftrightarrow x\in C(n,q)$. It follows that
\[E_{P,\mathcal{A}}(q)=\sum_{\Delta\in\mathscr{R}(\mathcal{A},q)}\big|\Delta\cap C(n,q)\big|, \sp \bar{E}_{P,\mathcal{A}}(q)=\sum_{\Delta\in\mathscr{R}(\mathcal{A},q)}\big|\bar{\Delta}\cap C(n,q)\big|\bar{\chi}(\mathcal{A},q).\]
From Theorem \ref{F_q}, it is easily seen that $E_{P,\mathcal{A}}(q)=\chi(\mathcal{A},q)$. Beck and Zaslavsky's reciprocity theorem \cite{Beck} states that $\bar{E}_{P,\mathcal{A}}(q)=(-1)^nE_{P,\mathcal{A}}(-q)$. It follows that $\bar{\chi}(\mathcal{A},q)=(-1)^n\chi(\mathcal{A},q)$, which completes the proof of Theorem \ref{comb-interpre} in the case that $\mathcal{A}$ is linear. We know that Beck and Zaslavsky's reciprocity theorem is a direct consequence of Ehrhart theory. So Theorem \ref{comb-interpre} can be viewed as an easy application of Ehrhart theory when the arrangement is linear. However, if the hyperplane arrangement is not linear, we can not find a way at the moment to interpret Theorem \ref{comb-interpre} as a consequence of Ehrhart theory or Beck and Zaslavsky's reciprocity theorem. We use an easy example to show the reasons. Let the arrangement $\mathcal{A}$ in $\Bbb{R}^2$  consist of three hyperplanes, $H_1: x=0, H_2: y=0,$ and $H_3: x+y=1$. Then the polytope $\Delta\in \mathscr{R}(\mathcal{A},q)$ bounded by $H_1, H_2,$ and $H_3$ is not dilated as $q$ changes. So Ehrhart theory can not be applied to $\Delta$.

We next prove Theorem \ref{comb-interpre}, without the hypothesis of linearity, by applying the convolution formula in Proposition \ref{prop}. For $\bm{x}\in \Bbb{R}^n$, use $\mathcal{A}_{\bm x}$ to denote the collection of all possible hyperplanes $H(kq)$ who pass through $\bm{x}$, i.e.,
\[\mathcal{A}_{\bm x}=\{H(kq)\mid H\in \mathcal{A}, k\in \Bbb{Z}, {\bm x}\in H(kq)\}.\]
Then $\mathcal{A}_{\bm x}$ is a central hyperplane arrangement in $\Bbb{R}^n$. Recall that $\mathscr{R}(\mathcal{A})$ denotes the collection of regions in $\Bbb{R}^n$ separated by all hyperplanes $H\in \mathcal{A}$, and $r(\mathcal{A})=\big|\mathscr{R}(\mathcal{A})\big|$.
\begin{lemma}\label{A(x)} With previous notations, we have
\[\#\{\Delta\in \mathscr{R}(\mathcal{A},q)\mid {\bm x}\in \bar{\Delta}\}=r(\mathcal{A}_{\bm x}), \sp \forall\, {\bm x}\in \Bbb{R}^n.\]
\end{lemma}
\begin{proof}
We shall prove it by constructing a bijection $\psi:\{\Delta\in \mathscr{R}(\mathcal{A},q)\mid {\bm x}\in \bar{\Delta}\}\to \mathscr{R}(\mathcal{A}_{\bm x})$ for any ${\bm x}\in \Bbb{R}^n$. It is obvious that $M(\mathcal{A},q)\subseteq M(\mathcal{A}_{\bm x})$. Then for all $\Delta \in \mathscr{R}(\mathcal{A},q)$ and $R\in \mathscr{R}(\mathcal{A}_{\bm x})$, we have either $\Delta \in R$ or $\Delta \cap R=\emptyset$, since $\Delta$ and $R$ are connected components of $M(\mathcal{A},q)$ and $M(\mathcal{A}_{\bm x})$ respectively. According to $\sqcup_{\Delta\in\mathscr{R}(\mathcal{A},q)}\Delta=M(\mathcal{A},q)\subseteq M(\mathcal{A}_{\bm x})=\sqcup_{R\in\mathscr{R}(\mathcal{A}_{\bm x})}R$, we can conclude that each $\Delta\in \mathscr{R}(\mathcal{A},q)$ is contained in a unique region $R\in \mathscr{R}(\mathcal{A}_{\bm x})$, denoted $R_\Delta$. It defines the map $\psi:\Delta\mapsto R_\Delta$. To prove $\psi$ is surjective, consider the set $\Sigma_R=\{\Delta\in \mathscr{R}(\mathcal{A},q)\mid \Delta\subseteq R\}$ for any $R\in \mathscr{R}(\mathcal{A}_{\bm x})$. Note that $R\cap \Delta=\emptyset$ for all $\Delta\in\mathscr{R}(\mathcal{A},q)-\Sigma_R$, and $R$ is an open set. Thus $R\cap \big(\cup_{\Delta\in\mathscr{R}(\mathcal{A},q)-\Sigma_R}\bar{\Delta}\big)=\emptyset$. Since $\cup_{\Delta\in\mathscr{R}(\mathcal{A},q)}\bar{\Delta}=\Bbb{R}^n$, we then obtain
\[
R-\cup_{\Delta\in \Sigma_R}\bar{\Delta}=R-\cup_{\Delta\in\mathscr{R}(\mathcal{A},q)}\bar{\Delta}=\emptyset.
\]
It implies that $\bar{R}=\cup_{\Delta\in\mathscr{R}(\mathcal{A},q)}\bar{\Delta}$. On the other hand, we obviously have ${\bm x}\in \bar{R}$ for $R\in \mathscr{R}(\mathcal{A}_{\bm x})$. Thus ${\bm x}\in \bar{\Delta}$ for some $\Delta\in \Sigma_R$. Namely, each region $R\in \mathscr{R}(\mathcal{A}_{\bm x})$ contains a region $\Delta\in \mathscr{R}(\mathcal{A},q)$ such that ${\bm x}\in \bar{\Delta}$, which proves the surjection of $\psi$. To show that $\psi$ is injective, suppose that we have $\Delta_1, \Delta_2\in \{\Delta\in \mathscr{R}(\mathcal{A},q)\mid {\bm x}\in \bar{\Delta}\}$ and $\Delta_1\neq \Delta_2$ such that $R=R_{\Delta_1}=R_{\Delta_2}$. Let $H(kq)$ be a separating hyperplane of $\Delta_1$ and $\Delta_2$, i.e., $\Delta_1\subseteq H(kq)^+$ and $\Delta_2\subseteq H(kq)^-$ \big(or, $\Delta_1\subseteq H(kq)^-$ and $\Delta_2\subseteq H(kq)^+$\big), where $H^+$ and $H^-$ are two closed half space of $\Bbb{R}^n$ divided by $H$. Then we have ${\bm x}\in \bar{\Delta}_1\cap \bar{\Delta}_2\subseteq  H(kq)^+\cap H(kq)^-=H(kq)$. It implies that $H(kq)\in \mathcal{A}_{\bm x}$. From the assumption $R\in \mathscr{R}(\mathcal{A}_{\bm x})$, we then have $R\cap H(kq)=\emptyset$. Notice that $R$ is connected. So $R\cap H(kq)^+=\emptyset$ or $R\cap H(kq)^-=\emptyset$. This contradicts to $\Delta_1, \Delta_2\subseteq R$ and $\Delta_1\subseteq H(kq)^+, \Delta_2\subseteq H(kq)^-$, which completes the proof.
\end{proof}
Let $\mathcal{A}$ and  $\mathcal{B}$ be two hyperplane arrangements in $\Bbb{R}^n$. We call $\mathcal{B}$ a \emph{translation} of $\mathcal{A}$ if there is a $t_H\in \Bbb{R}$ for each $H\in \mathcal{A}$ such that $\mathcal{B}=\{H(t_H)\mid H\in \mathcal{A}\}$.
\begin{lemma}\label{translation}
Suppose $\mathcal{A}$ and $\mathcal{B}$ are central hyperplane arrangements. If $\mathcal{B}$ is a translation of $\mathcal{A}$, then $r(\mathcal{B})=r(\mathcal{A})$.
\end{lemma}
\begin{proof}Since $\mathcal{A}$ and $\mathcal{B}$ are central, take ${\bm a}\in\cap_{H\in \mathcal{A}}H$ and ${\bm b}\in\cap_{H\in \mathcal{B}}H$. Consider the translation $\tau:{\bm x}\mapsto {\bm x}+{\bm b}-{\bm a}$ of $\Bbb{R}^n$. Since $\mathcal{B}$ is a translation of $\mathcal{A}$, then $\tau$ defines a bijection between $\mathcal{A}$ and $\mathcal{B}$, as well as a bijection between $L(\mathcal{A})$ and $L(\mathcal{B})$. So $r(\mathcal{B})=r(\mathcal{A})$.
\end{proof}
Given $H\in \mathcal{A}$, let $H_{(q)}$ be a subset of $C(n,q)$ defined by
\[H_{(q)}=C(n,q)\cap \big(\cup_{k\in \Bbb{Z}}H(kq)\big).\]
Then $\mathcal{A}_{(q)}=\{H_{(q)}\mid H\in \mathcal{A}\}$ forms an arrangement of sets in $C(n,q)$. Similar as before, we have the notations $L(\mathcal{A}_{(q)})$, $M(\mathcal{A}_{(q)})$, $\mathcal{A}_{(q)}|X_{(q)}$, $\mathcal{A}_{(q)}/ X_{(q)}$ for $X_{(q)}\in L(\mathcal{A}_{(q)})$, i.e.,
\begin{eqnarray*}
&&L(\mathcal{A}_{(q)})=\{\cap_{H\in \mathcal{B}}H_{(q)}\mid \mathcal{B}\subseteq \mathcal{A}\},\sp\sp\sp M(\mathcal{A}_{(q)})=C(n,q)-\cup_{H\in \mathcal{A}}H{(q)},\\
&&\mathcal{A}_{(q)}|X_{(q)}=\{H_{(q)}\mid H\in \mathcal{A}, X\subseteq H\}, \sp \mathcal{A}_{(q)}/ X_{(q)}=\{H_{(q)}\cap X_{(q)}\mid H\notin \mathcal{A}\ X_{(q)}\}.
\end{eqnarray*}
Consider the map
\[\rho: C(n,q)\to \Bbb{F}_q,\sp {\bm x}\mapsto {\bm x}\mod q.\]
It is obvious that $\rho$ is a bijection. Moreover, for any $\mathcal{B}\subseteq \mathcal{A}$, we have $\rho\big(\cap_{H\in \mathcal{B}}H_{(q)}\big)=\cap_{H\in \mathcal{B}}H_{q}$. So $\rho$ automatically induces an isomorphism of $L(\mathcal{A}_{(q)})$ and $L(\mathcal{A}_q)$ with $\rho(X_q)=X_{(q)}$. It is easy to see that $\big|X_q\big|=\big|X_{(q)}\big|$, and then $\big|M(\mathcal{A}_q)\big|=\big|M(\mathcal{A}_{(q)})\big|$. \\

\emph{Proof of Theorem {\rm \ref{comb-interpre}:}} Since that  
\[
M\big(\mathcal{A}_{(q)}/X_{(q)}\big)=\cap_{H\in \mathcal{A}\mid X}H_{(q)}-\cap_{H\notin \mathcal{A}\mid X}H_{(q)},
\]
then we have $C(n,q)=\sqcup_{X\in L(\mathcal{A})}M\big(\mathcal{A}_{(q)}/X_{(q)}\big)$. 
For any ${\bm x}\in M\big(\mathcal{A}_{(q)}/X_{(q)}\big)$, we can see that $\mathcal{A}_{\bm x}$ is a translation of $\mathcal{A}|X$. Since both $\mathcal{A}_{\bm x}$ and $\mathcal{A}|X$ are central, it follows by Lemma \ref{translation} that $r(\mathcal{A}_{\bm x})=r(\mathcal{A}|X)$. Applying Lemma \ref{A(x)}, we have
\begin{eqnarray*}
\sum_{\Delta\in \mathscr{R}(\mathcal{A},q)}\big|\bar{\Delta}\cap C(n,q)\big|=\sum_{{\bm x}\in C(n,q)}\#\{\Delta\mid {\bm x}\in \bar{\Delta}\}=\sum_{{\bm x}\in C(n,q)}r(\mathcal{A}_{\bm x})=\sum_{X\in L(\mathcal{A})}r(\mathcal{A}|X)\big|M\big(\mathcal{A}_{(q)}/X_{(q)}\big)\big|.
\end{eqnarray*}
Since $\big|M\big(\mathcal{A}_{(q)}/X_{(q)}\big)\big|=\big|M\big(\mathcal{A}_q/X_q\big)\big|$, it follows that 
\[\bar{\chi}(\mathcal{A},q)=\sum_{X\in L(\mathcal{A})}r(\mathcal{A}|X)\big|M\big(\mathcal{A}_q/X_q\big)\big|,\]
which completes the proof by Proposition \ref{prop}.

\section{Convolution Formulae on Tutte Polynomials}
In this section, we shall apply Theorem \ref{Mobius-conj} to formulate those convolution identities mentioned in \cite{Joseph2010,Staton1999}. Let $M$ be a matroid with the ground set $E$ and the rank function $r_{\tiny M}$. For simplicity, write $r_M(E)=r(M)$ for the rank of the matriod $M$. The rank generating function $R_M(x,y)$ and the Tutte polynomial $T_M(x,y)$ of $M$ are defined by
\[
R_M(x,y)=\sum_{A\subseteq E} x^{r(M)-r_M(A)}\, y^{|A|-r_M(A)},\sp T_M(x,y) = R_M(x - 1,y - 1).
\]
If $S$ is a subset of $E$, the restriction of $M$ to $S$, written as $M|S$, is the matroid on the ground set $S$ whose rank function is $r_{M|S}(A)=r_M(A)$ for all $A\subseteq S$. If $T$ is a subset of $E$, the contraction of $M$ by $T$, written as $M/T$, is the matroid on the ground set $E-T$ whose rank function is $r_{M/T}(A) = r_M(A \cup T)-r_M(T)$ for all $A\subseteq E-T$. With these definitions, we have
\begin{eqnarray*}
R_{M|S}(x,y)&=&\sum_{A\subseteq S}x^{r(M|S)-r_M(A)}\,y^{|A|-r_M(A)};\\
R_{M/T}(x,y)&=&\sum_{T\subseteq A\subseteq E}x^{r(M)-r_M(A)}\,y^{|A|-|T|-r_M(A)+r_M(T)}.
\end{eqnarray*}
To write the Tutte polynomial as the M\"{o}bius conjugation, consider the poset $(2^E,\subseteq)$ whose M\"{o}bius function is given by $\mu(A,B)=(-1)^{|A-B|}$ for all $A\subseteq B\subseteq E$. Let
\[
f(A)=(-x)^{r(M)-r_M(A)},\sp g(A)=(-y)^{|A|-r_M(A)},\sp \forall \; A\subseteq E.
\]
Then we have
\begin{eqnarray*}
\mu^\ast(fg)(\emptyset,E)=\sum_{A\subseteq E}\mu(\emptyset,A)\,(-x)^{r(M)-r_M(A)}\,(-y)^{|A|-r_M(A)}=(-1)^{r(M)}R_M(x,y).
\end{eqnarray*}
Similarly, we can obtain
\begin{eqnarray*}
(-1)^{r(M|S)}R_{M|S}(-1,y)&=&\mu^\ast(g)(\emptyset, S),\\
(-1)^{r(M/T)}R_{M/T}(x,-1)&=&\mu^\ast(f)(T,E).
\end{eqnarray*}
where the poset $(2^{E-T},\subseteq)$ for the last identity is identified with the interval $[T,E]$ of the poset $(2^E,\subseteq)$. Since $r(M)=r(M|A)+r(M/A)$ for all $S\subseteq E$ and $\mu^\ast(fg)=\mu^\ast(g)\ast\mu^\ast(f)$ by Theorem \ref{Mobius-conj}, we obtain
\[R_M(x,y)=\sum_{A\subseteq E}R_{M|A}(-1,y)\,R_{M/A}(x,-1).\]
This is equivalent to the main result obtained by W. Kook, V. Reiner, and D. Stanton \cite{Staton1999}.
\begin{theorem}{\rm \cite{Staton1999}} The Tutte polynomial $T_M(x,y)$ satisfies that
\[T_M(x,y)=\sum_{A\subseteq E}T_{M|A}(0,y)\,T_{M/A}(x,0).\]
\end{theorem}
Similar method can be applied to obtain the convolution identities of the subset-corank polynomial $SC_M({\bm x},\lambda)$ defined in \cite{Joseph2010}. Denote by $x_e$ the indeterminate indexed by $e\in E$. Given $A\subseteq E$, write $x_A$ for the monomial $\Pi_{e\in A}x_e$ and ${\bm x}$ for the collection of $x_e$ for all $e\in E$. The subset-corank polynomial $SC_M({\bm x},\lambda)$ is defined to be
\[
SC_M({\bm x},\lambda)=\sum_{A\subseteq E}x_A\,\lambda^{r(M)-r_M(A)}.
\]
From the definition, for any $S, T\subseteq E$, we have
\begin{eqnarray*}
SC_{M|S}({\bm x},\lambda)&=&\sum_{A\subseteq S}x_A\,\lambda^{r(M|S)-r_M(A)},\\
SC_{M/T}({\bm x},\lambda)&=&\sum_{T\subseteq A\subseteq E}x_{A-T}\,\lambda^{r(M)-r_M(A)}.
\end{eqnarray*}
Take $(2^E,\subseteq)$ to be the poset and let
\[
f(A)=(-x)_A=(-1)^{|A|}x_A, \sp g(A)=\lambda^{r(M)-r_M(A)},\sp \forall\; A\subseteq E.
\]
Then we can easily obtain
\[\mu^\ast(fg)(\emptyset, E)=SC_M({\bm x},\lambda).\]
On the other hand, for any $S,T\subseteq E$, we have
\begin{eqnarray*}
\mu^\ast(f)(\emptyset, S)&=&\sum_{A\subseteq S}x_A=(x+1)_A,\\
\mu^\ast(g)(T, E)&=&\sum_{T\subseteq A\subseteq E}(-1)^{|A-T|}\,\lambda^{r(M)-r_M(A)}=SC_{M/T}(\bm{-1},\lambda).
\end{eqnarray*}
So we have the following formula which is the identity 5 obtained in \cite{Joseph2010}.
\begin{theorem}{\rm \cite{Joseph2010}} The subset-corank polynomial $SC_M({\bm x},\lambda)$ satisfies
\[
SC_M({\bm x},\lambda)=\sum_{A\subseteq E}(x+1)_A \,SC_{M/T}(\bf{-1},\lambda).
\]
\end{theorem}
Similarly, the identity 1 in \cite{Joseph2010} can be easily obtained in this way. Let
\[
f(A)=x_A\,\lambda^{r(M)-r_M(A)},\sp g(A)=(-y)_A\,\xi^{r(M)-r_M(A)},\sp\forall\; A\subseteq E.
\]
Then we have
\[
\mu^\ast(fg)(\emptyset,E)=SC_M(\bm{xy},\lambda\xi),
\]
Where $\bm{ xy}$ means the collection of $x_ey_e$ for all $e\in E$. On the other hand, for any $S,T\subseteq E$, we have
\begin{eqnarray*}
\mu^\ast(f)(\emptyset, S)&=&\sum_{A\subseteq S}(-1)^{|A|}\,x_A\,\lambda^{r(M)-r_M(A)}=\lambda^{r(M)-r(M|S)}\,SC_{M|S}(\bm{-x},\lambda),\\
\mu^\ast(g)(T, E)&=&\sum_{T\subseteq A\subseteq E}(-1)^{|A-T|}\,(-y)_A\,\xi^{r(M)-r_M(A)}=(-y)_T\, SC_{M/T}(\bm{y},\xi).
\end{eqnarray*}
Applying Theorem \ref{Mobius-conj}, we have
\begin{theorem}{\rm \cite{Joseph2010}} The subset-corank polynomial $SC_{M}(\bm{x},\lambda)$ satisfies
\[
SC_M(\bm{xy},\lambda\xi)=\sum_{A\subseteq E}\lambda^{r(M)-r(M|S)}\,(-y)_T\, SC_{M|S}(\bm{-x},\lambda) \,SC_{M/T}(\bm{y},\xi).
\]
\end{theorem}
We remark that other identities in \cite{Joseph2010} and \cite{Joseph2004} can be obtained in a similar way as above three.\\

\begin{center}
{\bf Acknowledgements}
\end{center}
I express my gratitude to Prof. Joseph PeeSin Kung for his kindly help in the preparation of this work. He advised me providing a combinatorial proof for Theorem \ref{region-conv} and helped checking many details. I thank Prof. Christos A. Athanasiadis for his suggestions on the relevance of Theorem \ref{comb-interpre} with other reciprocity theorems, and thank Prof. Yeong-nan Yeh and Prof. Beifang Chen for many helpful comments on this work.

\end{document}